\documentclass{amsart}
\usepackage{amssymb}
\usepackage{amsfonts}
\usepackage{xcolor}
\usepackage[colorlinks=true, citecolor=red, linkcolor=blue]{hyperref}

\newtheorem{theorem}{Theorem}
\theoremstyle{plain}

\newtheorem{lemma}{Lemma}

\numberwithin{equation}{section}
\input{tcilatex}

\begin{document}
\title[On groups with cubic polynomial conditions]{On groups with cubic
polynomial conditions}
\author{A. Grishkov}
\address{Universidade Estadual de Sao Paulo, Sao Paulo, Brazil}
\email{grishkov@imeusp.br}
\author{R. Nunes }
\address{Universidade Federal de Goias, Goiania, Goias, Brazil}
\email{ricardo@gmail.com}
\author{S. Sidki}
\address{Departamento de Matem\'{a}tica, Universidade de Bras\'{\i}lia,
Brasilia, DF 70910-900, Brazil}
\email{ssidki@gmail.com}
\thanks{The first author thanks the University of Brasilia for visits during
2012-2014 and CNPq, FAPESP for support. The second author acknowledges
support from Procad-CAPES for post-doctoral studies. The third author thanks
Professor Eamonn O'Brien for discussions and hospitatlity at the University
of Auckland during November 2013.\\
All authors are grateful to the referees for constructive critisisms.}
\date{February, 2015}
\subjclass[2010]{ 20F05, 16S15, 16S34 }
\keywords{Cubic condictions on groups, Unipotent groups, Non-commutative
Groebner.}

\begin{abstract}
Let $F_{d}$ be the free group of rank $d$, \ freely generated by $\left\{
y_{1},...,y_{d}\right\} $, and let $\mathbb{D}F_{d}$ be the group ring over
an integral domain $\mathbb{D}$. Given a subset $E_{d}$ of $F_{d}$
containing the generating set, assign to each $s$ in $E_{d}$ a monic
polynomial $p_{s}\left( x\right)
=x^{n}+c_{s,n-1}x^{n-1}+...+c_{s,1}x+c_{s,0}\in \mathbb{D}\left[ x\right] $
and define the quotient ring 
\begin{equation*}
A\left( d,n,E_{d}\right) =\frac{\mathbb{D}F_{d}}{\left\langle p_{s}\left(
s\right) \mid s\in E_{d}\right\rangle _{ideal}}\text{.}
\end{equation*}%
When $p_{s}\left( s\right) $ is cubic for all $s$, we construct a finite set 
$E_{d}$ such that $A\left( d,n,E_{d}\right) $ has finite rank over an
extension of $\mathbb{D}$ by inverses of some of the coefficients of the
polynomials. When the polynomials are all equal to $(x-1)^{3}$ and $\mathbb{%
D=Z}\left[ \frac{1}{6}\right] $, we construct a finite subset $P_{d}$ of $%
F_{d}$ such that the quotient ring $A\left( d,3,P_{d}\right) $ has finite $%
\mathbb{D}$-rank and its augmentation ideal is nilpotent. The set $P_{2}$ is 
$\left\{
y_{1},y_{2},y_{1}y_{2},y_{1}^{-1}y_{2},y_{1}^{2}y_{2},y_{1}y_{2}^{2},\left[
y_{1},y_{2}\right] \right\} $ and we prove that $\left( x-1\right) ^{3}=0$
is satisfied by all elements in the image of $F_{2}$ in $A\left(
2,3,P_{d}\right) $.
\end{abstract}

\maketitle

\section{Introduction}

The impact of finite order conditions on a group has guided major
developments in group theory and so have similar finiteness questions in the
theory of algebras \cite{zelmanov}. The purpose of this paper is to examine
finitely generated groups where a finite number of its elements satisfy
polynomial equations in one variable in degrees $2$ and $3$.

The precise setting is as follows: we start with a free group $F_{d}$ of
rank $d$, \ freely generated by $\left\{ y_{1},...,y_{d}\right\} $, with the
group ring $\mathbb{D}F_{d}$ over an integral domain $\mathbb{D}$ and with a
set $E_{d}$ of elements of $F_{d}$ containing the generating set. We then
assign to each $s$ in $E_{d}$ a monic polynomial $p_{s}\left( x\right)
=x^{n}+c_{s,n-1}x^{n-1}+...+c_{s,1}x+c_{s,0}\in $ $\mathbb{D}\left[ x\right] 
$ and define the quotient ring 
\begin{equation*}
A\left( d,n,E_{d}\right) =\frac{\mathbb{D}F_{d}}{\left\langle p_{s}\left(
s\right) \mid s\in E_{d}\right\rangle _{ideal}}\text{.}
\end{equation*}%
If $c_{s,0}=0$ then $s$ satisfies a polynomial of degree less than $n$. If $%
c_{s,0}$ is not zero then we may replace $\mathbb{D}$ by its extension by
the inverse of $c_{s,0}$.

Let $a_{1},...,a_{d}$ be the respective images of $y_{1},...,y_{d}$ in $%
\mathbf{A}_{d}=A\left( d,n,E_{d}\right) $, let $G_{d}=\left\langle
a_{1},...,a_{d}\right\rangle $ and let $\mathbf{B}_{d}=\omega \left( \mathbf{%
A}_{d}\right) $ be the image of the augmentation ideal of $\mathbb{D}F_{d}$;
recall that $\mathbf{B}_{d}$ is generated by $u\left( g\right) =g-1$ for all 
$g\in G$.

First we prove a finiteness rank condition for $n=3.$

\begin{theorem}
Define the following subsets of $F_{d}$ 
\begin{equation*}
E_{1}=\left\{ y_{1}\right\} ,\text{ }M_{1}=\left\{ e,y_{1}^{\pm 1}\right\} 
\text{ }
\end{equation*}%
and inductively for $1\leq s$ $\leq $ $d-1$, 
\begin{eqnarray*}
E_{s+1} &=&E_{s}\cup M_{s}y_{s+1}^{\pm 1}, \\
M_{s+1} &=&M_{s}\cup M_{s}y_{s+1}^{\pm 1}M_{s} \\
&&\cup \text{ }M_{s}y_{s+1}^{-1}\left( M_{s}\backslash \left\{ e\right\}
\right) y_{s+1}M_{s}\text{.}
\end{eqnarray*}%
Suppose that each $s\in E_{d}$ satisfies some cubic polynomial $p_{s}\left(
x\right) =x^{3}+c_{s,2}x^{2}+c_{s,1}x+c_{s,0}\in $ $\mathbb{D}\left[ x\right]
$. Then (i) $A\left( d,n,E_{d}\right) $ is the linear span of the images of $%
M_{d}$ over an extension of $\mathbb{D}$ by inverses of some of the
coefficients; (ii) the set $E_{2}$ is $\left\{
y_{1},y_{2},y_{1}y_{2},y_{1}y_{2}^{-1}\right\} $ and for general $d$, the
elements of $E_{d}$ are primitive in the generators $\left\{
y_{1},...,y_{d}\right\} $; (iii) $\left\vert M_{2}\right\vert \leq 39$ and $%
\log _{3}\left\vert M_{d}\right\vert \leq 4.3^{d-2}$ for all $d$.
\end{theorem}

Since the $d$-generated Burnside group $B\left( d,3\right) $ of exponent $3$
satisfies the unipotent condition $\left( x-1\right) ^{3}=0$ in its group
algebra over $GF\left( 3\right) $, the mentioned logarithmic upper bound is
at least $\log _{3}\left\vert B\left( d,3\right) \right\vert =d+\binom{d}{2}+%
\binom{d}{3}$ (\cite{V-L}, page 88).

The rest of the paper is a study of \ the quotient rings $A\left(
d,n,E_{d}\right) $ for unipotent polynomials $p_{s}\left( x\right) =\left(
x-1\right) ^{n}$ where $n=2,3$.

In case $n=2$ we prove:

\begin{theorem}
Define the quotient ring $\mathbf{A}_{d}=A\left( d,2,S_{d}\right) $ $=\frac{%
\mathbb{Z}F_{d}}{\left\langle \left( x-1\right) ^{2}\mid x\in
S_{d}\right\rangle }$ where 
\begin{equation*}
S_{d}=\left\{ y_{i}\text{ }\left( 1\leq i\leq d\right) ,y_{i}y_{j}\text{ }%
\left( 1\leq i<j\leq d\right) \right\} \text{.}
\end{equation*}%
Then: (i) the map%
\begin{eqnarray*}
\varphi &:&a_{1}\rightarrow a_{1,d}=\left( 
\begin{array}{cc}
I_{2^{d-1}} & 0 \\ 
I_{2^{d-1}} & I_{2^{d-1}}%
\end{array}%
\right) , \\
a_{i} &\rightarrow &a_{i,d}=\left( 
\begin{array}{cc}
a_{i-1,d-1} & 0 \\ 
a_{i-1,d-1} & \left( a_{i-1,d-1}\right) ^{-1}%
\end{array}%
\right) \text{ }\left( 2\leq i\leq d\right)
\end{eqnarray*}%
extends to a \ monomorphism $\varphi :$ $\mathbf{A}_{d}\rightarrow M(2^{d},%
\mathbb{Z)}$ and $\mathbf{A}_{d}$ is a free $\mathbb{Z}$-module of rank $%
2^{d}$; (ii) $\mathbf{B}_{d}$ is nilpotent of degree $d+1,$ $\left( \mathbf{B%
}_{d}\right) ^{d}=\mathbb{Z}u\left( a_{1}\right) ...u\left( a_{d}\right) $;
(iii) $G_{d}$ is a free $d$-generated nilpotent group of class $2$.
\end{theorem}

In case $n=3$, first we prove:

\begin{theorem}
Let $\mathbb{D=Z}\left[ \frac{1}{6}\right] $. Consider the quotient ring $%
\mathbf{A}$ $=A\left( 2,3,P\right) =\frac{\mathbb{D}F_{2}}{\left\langle
\left( x-1\right) ^{3}\mid x\in P\right\rangle }$ where 
\begin{equation*}
P=\left\{
y_{1},y_{2},y_{1}y_{2},y_{1}^{-1}y_{2},y_{1}^{2}y_{2},y_{1}y_{2}^{2},\left[
y_{1},y_{2}\right] \right\} \text{,}
\end{equation*}%
Then \newline
(i) rank $\mathbf{A}$ $=18,$\newline
(ii) the ideal $\mathbf{B}$ is nilpotent of degree $6,$\newline
(iii) $G$ is a free $2$-generated nilpotent group of degree $4$,\newline
(iv) $(x-1)^{3}=0$ is satisfied by all elements of $G$.
\end{theorem}

The development of relations in the proof of Theorem 3 is an abridged form
of that presented in \cite{GOS} where we traced the effect of gradually
introducing relations to $\mathbb{Z}F_{2}$ and also detected the appearance
of torsion conditions in the ring. There, we calculated relations by hand
and confirmed them by using the computer program GAP \cite{gap}. In
particular, we used the package GBNP \cite{gbnp} along with some special
substitution routines \cite{ricardo}. If the analysis were carried further,
it would provide a complete description of $\frac{\mathbb{Z}F_{2}}{%
\left\langle \left( x-1\right) ^{3}\mid x\in P\right\rangle }$. In the
present paper we take the shorter route by analyzing directly the ring $%
\mathbf{A}$. In certain places, when the computational procedure becomes
clear, we refer to \cite{GOS} for more details. Item (iv) of the theorem
follows from a determination of the set of solutions of $X^{3}=0$ in $%
\mathbf{B}$, as an algebraic affine variety over $\mathbb{D}$.

It follows from a deep theorem of Zelmanov on Lie algebras satisfying an $n$%
-Engel identity over a field $k$ of characteristic zero, that if all the
elements of a finitely generated group $H$ satisfy $\left( x-1\right) ^{n}=0$
then $H$ is nilpotent (\cite{vovsi}, page 71-72). This result is also true
for fields $k$ of finite characteristic $p$ which is large enough. The lower
bound for $p$ provided by the proof is super-exponential in $n$; it was
conjectured by Zelmanov that $p\geq 2n$ is sufficient.

We conclude the paper with:

\begin{theorem}
Let $\mathbb{D=Z}\left[ \frac{1}{6}\right] $. There exists a finite subset $%
P_{d}$ of $F_{d}$, extending $P$, such that the quotient ring $\mathbf{A}%
_{d}=A\left( d,3,P_{d}\right) =\frac{\mathbb{D}F_{d}}{\left\langle \left(
x-1\right) ^{3}\mid x\in P_{d}\right\rangle }$ has finite $\mathbb{D}$-rank
and the augmentation ideal $\mathbf{B}_{d}$ is nilpotent.
\end{theorem}

We hope that the approach we have taken in this paper may be applied to
Burnside groups of exponent $5$ satisfying $(x-1)^{4}=0$ in characteristic $%
5 $. For recent works on this topic, see \cite{abd}, \cite{traus}.

\section{General Cubic Conditions (Proof of Theorem 1)}

(i) We treat first the case $d=2$. Write $a_{1}=a,a_{2}=b$. We may assume
each $x\in \left\{ a,b,ab,a^{-1}b\right\} $ satisfies%
\begin{equation*}
x^{3}=\delta _{x}x^{2}+\gamma _{x}x+\varepsilon _{x}
\end{equation*}%
for some $\gamma _{x},\delta _{x}$ $\varepsilon _{x}$ in $\mathbb{D}$ and $%
\varepsilon _{x}$ invertible in $\mathbb{D}$.

Multiplying the above by $x^{-1}$ produces%
\begin{equation*}
x^{2}=\varepsilon _{x}x^{-1}+\delta _{x}x+\gamma _{x},
\end{equation*}%
and by $x^{-3}$ produces%
\begin{equation*}
x^{-3}=-\varepsilon _{x}^{-1}\left( \gamma _{x}x^{-2}+\delta
_{x}x^{-1}-1\right) \text{;}
\end{equation*}%
that is, $x^{-1}$ satisfies a cubic polynomial over $\mathbb{D}$. More
generally, $x^{k}$ is a linear combination of $\left\{ 1,x,x^{-1}\right\} $
for all integers $k$. Therefore, the ring generated by $G$ is a $\mathbb{D}$%
-linear combination of monomials in $a^{i},b^{i}$ $\left( i=0,\pm 1\right) $%
. Conjugation of%
\begin{equation*}
\left( ab\right) ^{3}=\delta _{ab}\left( ab\right) ^{2}+\gamma _{ab}\left(
ab\right) +\varepsilon _{ab}
\end{equation*}%
by $a^{-1}$ produces%
\begin{equation*}
\left( ba\right) ^{3}=\delta _{ab}\left( ba\right) ^{2}+\gamma _{ab}\left(
ba\right) +\varepsilon _{ab}
\end{equation*}%
and inversion produces 
\begin{equation*}
\left( a^{-1}b^{-1}\right) ^{3}=-\varepsilon _{ab}^{-1}\left( \gamma
_{ab}\left( a^{-1}b^{-1}\right) ^{2}+\delta _{ab}\left( a^{-1}b^{-1}\right)
-1\right) \text{.}
\end{equation*}%
Similarly, the formula for $\left( a^{-1}b\right) ^{3}$ produces

\begin{equation*}
\left( ab^{-1}\right) ^{3}=-\varepsilon _{a^{-1}b}\left( \gamma
_{a^{-1}b}\left( ab^{-1}\right) ^{2}+\delta _{a^{-1}b}\left( ab^{-1}\right)
-1\right) \text{. }
\end{equation*}%
That is, every $a^{i}b^{j}$ $\left( i,j=\pm 1\right) $ satisfies a cubic
polynomial. On multiplying the formula%
\begin{equation*}
\left( ab\right) ^{3}=\delta _{ab}\left( ab\right) ^{2}+\gamma
_{ab}ab+\varepsilon _{ab}
\end{equation*}%
by $a^{-1}$ on the left and by $\left( ab\right) ^{-1}$ on the right, we
obtain%
\begin{equation*}
bab=\delta _{ab}b+\gamma _{ab}a^{-1}+\varepsilon _{ab}a^{-1}b^{-1}a^{-1}%
\text{.}
\end{equation*}%
On substituting in the expressions for $a^{-1},b^{-1}$ we obtain a formula
for $bab$ as a linear combination of $a^{i}b^{j}a^{k}$ $\left( i,j,k=0,\pm
1\right) $.

Similarly, from the formulas for $\left( a^{\varepsilon }b^{\delta }\right)
^{3}$ $\left( \varepsilon ,\delta =-1,1\right) $ we obtain $b^{m}a^{l}b^{m}$ 
$\left( l,m=1,2\right) $ as a linear combination of $a^{i}b^{j}a^{k}$ $%
\left( i,j,k=0,\pm 1\right) $. Therefore $b^{m}a^{l}b^{m+1}$ is a sum of
monomials $a^{i}b^{j}a^{k}b$ $\left( i,j,k=0,\pm 1\right) $, from which set
we can exclude $a^{i}ba^{k}b$ $\left( k=\pm 1\right) $.

We assert that a monomial of type $w=\alpha _{1}b^{j}\alpha _{2}b^{k}\alpha
_{3}b^{l}\alpha _{4}$ where $\alpha _{i}$ $=a^{\pm 1}$ $\left( 1\leq i\leq
4\right) $ and having $b$-syllable length $\left\vert j\right\vert
+\left\vert k\right\vert +\left\vert l\right\vert \leq 3$ is a sum of
monomials of type $\alpha _{1}b^{r}\alpha _{2}b^{s}\alpha _{3}$ with $%
r,s=\pm 1$ and $\left( r,s\right) \not=\left( 1,-1\right) $. \ From our
formulas, we may assume that $j\not=k\not=l$. Typically, $w=\alpha
_{1}bab^{-1}\alpha _{2}b\alpha _{3}$. Then,%
\begin{eqnarray*}
\alpha _{1}bab^{-1}\alpha _{2}b\alpha _{3} &=&\alpha _{1}\left(
bab^{-1}\right) \alpha _{2}b\alpha _{3}=\alpha _{1}\left( \sum
a^{i}b^{j}a^{k}b\right) \alpha _{2}b\alpha _{3}= \\
&=&\sum \alpha _{1}a^{i}.b^{j}a^{k}\left( b\alpha _{2}b\right) \alpha _{3} \\
&=&\sum \alpha _{1}^{\prime }b^{j}a^{k}\left( \sum \alpha _{2}^{\prime
}b^{j^{\prime }}\alpha _{2}^{\prime \prime }\right) \alpha _{3}=\sum \alpha
_{1}^{\prime }b^{j}a^{k^{\prime }}b^{j^{\prime \prime }}\alpha _{3}\text{.}
\end{eqnarray*}

For the general case, we argue by induction on $d$. Here, we need to
consider only monomials of the form $%
w_{1}a_{1}^{l_{1}}w_{2}a_{2}^{l_{2}}...w_{d}a_{d}^{l_{d}}$ where $l_{i}=\pm
1 $ and $w_{i}\in M_{d-1}$. The proof proceeds as before: $%
a_{d}^{m}wa_{d}^{m}$ $\left( m=\pm 1\right) $ is a sum of monomials $%
w^{i}a_{d}^{m}w^{k}$ $\left( i,k=0,1,2\right) $ and $a_{d}^{m}wa_{d}^{m+1}$
is a sum of monomials $w^{i}a_{d}^{m}w^{k}a_{d}$ from which we can exclude $%
w^{i}a_{d}w^{k}a_{d}$.

(ii) The assertion here is clear.

(iii) That $\left\vert M_{2}\right\vert \leq 39$ is also clear. For general $%
d\,$, the cardinality of $\left\vert M_{d}\right\vert $ $\ $is estimated as
follows%
\begin{eqnarray*}
\left\vert M_{s+1}\right\vert &\leq &\left\vert M_{s}\right\vert +\left\vert
M_{s}\right\vert ^{2}+\left\vert M_{s}\right\vert ^{3} \\
&\leq &\frac{\left\vert M_{s}\right\vert }{\left\vert M_{s}\right\vert -1}%
\left( \left\vert M_{s}\right\vert ^{3}-1\right)
\end{eqnarray*}%
and by induction on $d$, 
\begin{equation*}
\log _{3}\left\vert M_{d}\right\vert \leq 4.3^{d-2}\text{.}
\end{equation*}

\section{Quadratic Unipotent conditions (Proof of Theorem 2)}

(0) Let $H$ be a subgroup of a ring $S$ with unity and let $x\in H$ satisfy $%
u\left( x\right) ^{2}=0$. Then 
\begin{equation*}
u\left( x^{k}\right) =ku\left( x\right) \text{ for all integers }k\text{.}
\end{equation*}%
\ Suppose $x,y\in H$ satisfy 
\begin{equation*}
u\left( x\right) ^{2}=u\left( y\right) ^{2}=u\left( xy\right) ^{2}=0\text{.}
\end{equation*}%
Then%
\begin{eqnarray*}
u\left( \left( xy\right) ^{-1}\right) &=&-u\left( yx\right) =-\left( u\left(
y\right) +u\left( x\right) +u\left( y\right) u\left( x\right) \right) , \\
u\left( x^{-1}y^{-1}\right) &=&u\left( x^{-1}\right) +u\left( y^{-1}\right)
+u\left( x^{-1}\right) u\left( y^{-1}\right) \\
&=&-u\left( x\right) -u\left( y\right) +u\left( x\right) u\left( y\right) ,
\\
u\left( y\right) u\left( x\right) &=&-u\left( x\right) u\left( y\right) 
\text{;}
\end{eqnarray*}%
\begin{eqnarray*}
u\left( \left[ x,y\right] \right) &=&u\left( x^{-1}y^{-1}xy\right) \\
&=&u\left( x^{-1}\right) +u\left( y^{-1}\right) +u\left( x\right) +u\left(
y\right) \\
&&+u\left( x^{-1}\right) u\left( y^{-1}\right) +u\left( x^{-1}\right)
u\left( x\right) +u\left( x^{-1}\right) u\left( y\right) \\
&&+u\left( y^{-1}\right) u\left( x\right) +u\left( y^{-1}\right) u\left(
y\right) \\
&&+u\left( x\right) u\left( y\right) \\
&=&u\left( x\right) u\left( y\right) -u\left( y\right) u\left( x\right)
=2u\left( x\right) u\left( y\right) \text{;}
\end{eqnarray*}%
\begin{equation*}
u\left( \left[ x,y\right] \right) ^{2}=0\text{.}
\end{equation*}%
(i) In $A_{d}$, 
\begin{equation*}
a_{i}^{2}=2a_{i}-1,
\end{equation*}%
for all $i$ and 
\begin{eqnarray*}
u\left( a_{j}\right) u\left( a_{i}\right) &=&-u\left( a_{i}\right) u\left(
a_{j}\right) , \\
a_{j}a_{i} &=&-a_{i}a_{j}+2a_{i}+2a_{j}-2
\end{eqnarray*}%
for all $i,j$. Therefore, $\mathbf{A}_{d}$ is a $\mathbb{Z}$-linear
combination of $1$ together with the monomials $%
a_{i_{1}}a_{i_{2}}...a_{i_{s}}$ for $1\leq i_{1}<i_{2}<...<i_{s}\leq d$. We
note that $\mathbf{A}_{d}$ maps onto the group ring $\mathbb{Z}\frac{F}{%
F^{\prime }}$ and therefore $rank_{\mathbb{Z}}\left( \mathbf{A}_{d}\right)
=2^{d}$.

(ii) Given the map 
\begin{eqnarray*}
\varphi &:&a_{1}\rightarrow a_{1,d}=\left( 
\begin{array}{cc}
I_{2^{d-1}} & 0 \\ 
I_{2^{d-1}} & I_{2^{d-1}}%
\end{array}%
\right) , \\
a_{i} &\rightarrow &a_{i,d}=\left( 
\begin{array}{cc}
a_{i-1,d-1} & 0 \\ 
a_{i-1,d-1} & \left( a_{i-1,d-1}\right) ^{-1}%
\end{array}%
\right) \text{ }\left( 2\leq i\leq d\right) \text{,}
\end{eqnarray*}%
we calculate 
\begin{eqnarray*}
\left( a_{i,d}\right) ^{-1} &=&\left( 
\begin{array}{cc}
\left( a_{i-1,d-1}\right) ^{-1} & 0 \\ 
-a_{i-1,d-1} & a_{i-1,d-1}%
\end{array}%
\right) , \\
a_{i,d}+\left( a_{i,d}\right) ^{-1} &=&\left( 
\begin{array}{cc}
a_{i-1,d-1}+\left( a_{i-1,d-1}\right) ^{-1} & 0 \\ 
0 & a_{i-1,d-1}+\left( a_{i-1,d-1}\right) ^{-1}%
\end{array}%
\right) \\
&=&2I_{2^{d}}\text{.}
\end{eqnarray*}%
By induction on $d$, we have%
\begin{equation*}
\left( a_{i,d}-I_{2^{d}}\right) ^{2}=\left( 
\begin{array}{cc}
\left( a_{i-1,d-1}-I_{2^{d-1}}\right) ^{2} & 0 \\ 
a_{i-1,d-1}\left( a_{i-1,d-1}+\left( a_{i-1,d-1}\right)
^{-1}-2I_{2^{d-1}}\right) & \left( \left( a_{i-1,d-1}\right)
^{-1}-I_{2^{d-1}}\right) ^{2}%
\end{array}%
\right) =0\text{.}
\end{equation*}%
Furthermore,%
\begin{eqnarray*}
a_{i,d}a_{j,d} &=&\left( 
\begin{array}{cc}
a_{i-1,d-1} & 0 \\ 
a_{i-1,d-1} & \left( a_{i-1,d-1}\right) ^{-1}%
\end{array}%
\right) \left( 
\begin{array}{cc}
a_{j-1,d-1} & 0 \\ 
a_{j-1,d-1} & \left( a_{j-1,d-1}\right) ^{-1}%
\end{array}%
\right) \\
&=&\left( 
\begin{array}{cc}
a_{i-1,d-1}a_{j-1,d-1} & 0 \\ 
\left( a_{i-1,d-1}+\left( a_{i-1,d-1}\right) ^{-1}\right) a_{j-1,d-1} & 
\left( a_{i-1,d-1}\right) ^{-1}\left( a_{j-1,d-1}\right) ^{-1}%
\end{array}%
\right) \\
&=&\left( 
\begin{array}{cc}
a_{i-1,d-1}a_{j-1,d-1} & 0 \\ 
2a_{j-1,d-1} & \left( a_{i-1,d-1}\right) ^{-1}\left( a_{j-1,d-1}\right) ^{-1}%
\end{array}%
\right) \\
a_{i,d}a_{j,d}-I_{2^{d}} &=&\left( 
\begin{array}{cc}
\left( a_{i-1,d-1}a_{j-1,d-1}-I_{2^{d-1}}\right) & 0 \\ 
2a_{j-1,d-1} & -a_{j-1,d-1}a_{i-1,d-1}+I_{2^{d-1}}%
\end{array}%
\right)
\end{eqnarray*}%
Therefore,%
\begin{eqnarray*}
\left( a_{i,d}a_{j,d}-I_{2^{d}}\right) ^{2} &=&\left( 
\begin{array}{cc}
\left( a_{i-1,d-1}a_{j-1,d-1}-I_{2^{d-1}}\right) ^{2} & 0 \\ 
\zeta & \left( -a_{j-1,d-1}a_{i-1,d-1}+I_{2^{d-1}}\right) ^{2}%
\end{array}%
\right) , \\
\zeta &=&2a_{j-1,d-1}\left( a_{i-1,d-1}a_{j-1,d-1}-I_{2^{d-1}}\right)
+\left( -a_{j-1,d-1}a_{i-1,d-1}+I_{2^{d-1}}\right) 2a_{j-1,d-1} \\
&=&0\text{,} \\
\left( a_{i,d}a_{j,d}-I_{2^{d}}\right) ^{2} &=&0\text{.}
\end{eqnarray*}%
Thus $\varphi $ extends to a ring homomorphism $\mathbf{A}_{d}\rightarrow
M(2^{d},\mathbb{Z)}$.

Observe that 
\begin{equation*}
\mathbf{A}_{d}=\mathbb{Z}\left\langle a_{1},...,a_{d-1}\right\rangle +%
\mathbb{Z}\left\langle a_{1},...,a_{d-1}\right\rangle a_{d}
\end{equation*}%
which will be shown to be a direct sum. We do this first in the
representation $\left( \mathbf{A}_{d}\right) ^{\varphi }$.

\begin{lemma}
Let $T_{0}=\mathbb{Z}$ and for $s\geq 1,T_{s}=\left( \mathbf{A}_{s}\right)
^{\varphi }$ and $U_{s+1}=T_{s}$ $\otimes I_{2^{s+1}}$. Then, $U_{s+1}\cong
T_{s}$ and 
\begin{equation*}
T_{d}=U_{d}\oplus U_{d}\alpha _{d,d}=U_{d}\oplus U_{d}\left( \alpha
_{d,d}\right) ^{-1}\text{.}
\end{equation*}
\end{lemma}

\begin{proof}
Let $d=1$. Then, $T_{1}=\mathbb{Z}\left( 
\begin{array}{cc}
1 & 0 \\ 
0 & 1%
\end{array}%
\right) \oplus \mathbb{Z}\left( 
\begin{array}{cc}
1 & 0 \\ 
1 & 1%
\end{array}%
\right) $. In general we have 
\begin{equation*}
T_{d}=U_{d}+U_{d}a_{d,d}
\end{equation*}%
and since $\left( \alpha _{d,d}\right) ^{-1}=2-\alpha _{d,d}$,%
\begin{equation*}
T_{d}=U_{d}+U_{d}\left( a_{d,d}\right) ^{-1}\text{.}
\end{equation*}%
Let $X,Y\in U_{d}$ \ and $\lambda ,\mu \in \mathbb{Z}$. Then, using the forms%
\begin{equation*}
X=\left( 
\begin{array}{cc}
X_{1} & 0 \\ 
V_{1} & \overline{X_{1}}%
\end{array}%
\right) ,Y=\left( 
\begin{array}{cc}
Y_{1} & 0 \\ 
W_{1} & \overline{Y_{1}}%
\end{array}%
\right) ,a_{d,d}=\left( 
\begin{array}{cc}
a_{d-1,d-1} & 0 \\ 
a_{d-1,d-1} & \left( a_{d-1,d-1}\right) ^{-1}%
\end{array}%
\right) \text{,}
\end{equation*}%
we expand%
\begin{equation*}
\lambda X+\mu Ya_{d,d}=\left( 
\begin{array}{cc}
\lambda X_{1}+\mu Y_{1}a_{d-1,d-1} & 0 \\ 
\lambda V_{1}+\mu \left( W_{1}+\overline{Y_{1}}\right) a_{d-1,d-1} & \lambda 
\overline{X_{1}}+\mu \overline{Y_{1}}\left( a_{d-1,d-1}\right) ^{-1}%
\end{array}%
\right) \text{. }
\end{equation*}%
Suppose%
\begin{equation*}
\lambda X+\mu Ya_{d,d}=0\text{;}
\end{equation*}%
we want to show $\lambda X=\mu Y=0$. Now 
\begin{eqnarray*}
\lambda X_{1}+\mu Y_{1}a_{d-1,d-1} &=&0,\text{ }\lambda \overline{X_{1}}+\mu 
\overline{Y_{1}}\left( a_{d-1,d-1}\right) ^{-1}=0, \\
\lambda V_{1}+\mu \left( W_{1}+\overline{Y_{1}}\right) a_{d-1,d-1} &=&0
\end{eqnarray*}%
and by induction,%
\begin{eqnarray*}
\lambda X_{1} &=&\mu Y_{1}=\lambda \overline{X_{1}}=\mu \overline{Y_{1}}=0%
\text{,} \\
\lambda V_{1} &=&\mu \left( W_{1}+\overline{Y_{1}}\right) =0\text{,} \\
\mu W_{1} &=&0\text{.}
\end{eqnarray*}
\end{proof}

Thus 
\begin{eqnarray*}
\mathbf{A}_{d} &=&\mathbb{Z}\left\langle a_{1},...,a_{d-1}\right\rangle +%
\mathbb{Z}\left\langle a_{1},...,a_{d-1}\right\rangle a_{d}\text{,} \\
\left( \mathbf{A}_{d}\right) ^{\varphi } &=&\left( \mathbb{Z}\left\langle
a_{1},...,a_{d-1}\right\rangle \right) ^{\varphi }+\left( \mathbb{Z}%
\left\langle a_{1},...,a_{d-1}\right\rangle \right) ^{\varphi }a_{d,d} \\
&=&U_{d}\oplus U_{d}\alpha _{d,d}\text{.}
\end{eqnarray*}%
By induction, the ring $\mathbf{A}_{d-1}$ and the subring $\mathbb{Z}%
\left\langle a_{1},...,a_{d-1}\right\rangle $ of $\mathbf{A}_{d}$ are free $%
\mathbb{Z}$-modules of rank $2^{d-1}$ isomorphic to $U_{d}$. Therefore, from
the above lemma, we reach $\mathbf{A}_{d}$ is a free $\mathbb{Z}$-module of
rank $2^{d}$ and conclude $\varphi $ is a monomorphism.

(iii) The ring $B_{d}$ is a free $\mathbb{Z}$-module freely generated by 
\begin{equation*}
\left\{ u\left( a_{j_{1}}\right) ...u\left( a_{j_{t}}\right) \mid 1\leq
j_{1}<...<j_{t}\leq d\right\}
\end{equation*}%
and $\left( B_{d}\right) ^{d+1}=0$, 
\begin{equation*}
\left( B_{d}\right) ^{d}=\mathbb{Z}u\left( a_{1}\right) ...u\left(
a_{d}\right) \not=0\text{.}
\end{equation*}

(iv) Since $u\left( \left[ a_{i},a_{j}\right] \right) ^{2}=0,u\left(
a_{l}\right) ^{2}=0$, it follows that for all $i,j,l$, 
\begin{eqnarray*}
u\left( \left[ \left[ a_{i},a_{j}\right] ,a_{l}\right] \right) &=&u\left( %
\left[ a_{i},a_{j}\right] \right) u\left( a_{l}\right) -u\left( a_{l}\right)
u\left( \left[ a_{i},a_{j}\right] \right) \\
&=&2u\left( a_{i}\right) u\left( a_{j}\right) u\left( a_{l}\right) -2u\left(
a_{l}\right) u\left( a_{i}\right) u\left( a_{j}\right) \\
&=&2u\left( a_{i}\right) u\left( a_{j}\right) u\left( a_{l}\right) -2u\left(
a_{i}\right) u\left( a_{j}\right) u\left( a_{l}\right) \\
&=&0\text{,} \\
\left[ a_{i},a_{j},a_{l}\right] &=&e\text{.}
\end{eqnarray*}%
Hence, $\gamma _{3}\left( G\right) =\left\{ e\right\} $. As we noted before $%
\mathbf{A}_{d}$ maps onto $\mathbb{Z}\frac{F}{F^{\prime }}$ and therefore $%
\frac{G_{d}}{G_{d}^{\prime }}$ is a free abelian group of rank $d$. We need
to show that $G_{d}^{\prime }$ is free abelian of rank $\binom{d}{2}$.
Clearly $G_{d}^{\prime }$ is an abelian group generated by $\left[
a_{i},a_{j}\right] $ for $1\leq i<j\leq d$. Since $\left[ a_{i},a_{j}\right]
=1+2u\left( a_{i}\right) u\left( a_{j}\right) $ and the set $\left\{ u\left(
a_{i}\right) u\left( a_{j}\right) \mid 1\leq i<j\leq d\right\} $ is $\mathbb{%
Z}$-linearly independent, it follows $G_{d}$ is a free $d$-generated
nilpotent group of class $2$.

\ \ \ \ 

\section{Cubic Unipotent Conditions}

Let $R$ be a ring with unity and let $G=\left\langle a_{i}~\left( 1\leq
i\leq d\right) \right\rangle $ be a multiplicative subgroup of $R$. Then for 
$g_{1},g_{2},....,g_{n}\in G$, 
\begin{equation*}
u(g_{1}g_{2}....g_{n})=\sum\limits_{1\leq i_{1}\,<i_{2}<...<i_{s}\leq
n}u(g_{i_{1}})u(g_{i_{2}})...u(g_{i_{s}})\text{.}
\end{equation*}

Suppose $\left( g-1\right) ^{3}=0$. Then, for $m\geq 1$,

\begin{eqnarray*}
g^{m} &=&\left( 1+u\left( g\right) \right) ^{m}=1+mu\left( g\right) +\binom{m%
}{2}u\left( g\right) ^{2}, \\
u\left( g^{m}\right) &=&g^{m}-1=mu\left( g\right) +\binom{m}{2}u\left(
g\right) ^{2}, \\
u\left( g^{m}\right) ^{2} &=&m^{2}u\left( g\right) ^{2},u\left( g^{m}\right)
^{3}=0\text{.}
\end{eqnarray*}%
For negative exponents, we use%
\begin{equation*}
g^{-1}=g^{2}-3g+3,
\end{equation*}%
\begin{eqnarray*}
u(g^{-1}) &=&u(g^{2})-3u(g)=-u(g)+u(g)^{2}, \\
u(g^{-1})^{2} &=&u(g)^{2}\text{.}
\end{eqnarray*}%
Thus, for $m\geq 1,$%
\begin{eqnarray*}
u(g^{-m}) &=&mu(g^{-1})+\binom{m}{2}u(g^{-1})^{2}=m\left(
-u(g)+u(g)^{2}\right) +\binom{m}{2}u(g)^{2} \\
&=&-mu(g)+\binom{m+1}{2}u(g)^{2}\text{.}
\end{eqnarray*}%
Using $\binom{-i+1}{2}=\frac{\left( -i+1\right) \left( -i\right) }{2}=\frac{%
\left( i-1\right) i}{2}=\binom{i}{2}$, we conclude

\begin{equation*}
u(g^{m})=mu(g)+\binom{m}{2}u(g)^{2}
\end{equation*}%
for all $m$.

\subsection{The $2$-generated case}

We analyze directly the ring 
\begin{equation*}
\mathbf{A}=\frac{\mathbb{D}F_{2}}{\left\langle \left( x-1\right) ^{3}\mid
x\in P\right\rangle },P=\left\{
y_{1},y_{2},y_{1}y_{2},y_{1}^{-1}y_{2},y_{1}^{2}y_{2},y_{1}y_{2}^{2},\left[
y_{1},y_{2}\right] \right\} \text{.}
\end{equation*}

Note that $\left( y_{1}^{-1}y_{2}-1\right) ^{3}=0$ implies $\left(
y_{1}y_{2}^{-1}-1\right) ^{3}=0$. We denote $u\left( a\right) =a-1,u\left(
b\right) =b-1$ in $\mathbf{B}$ by $U,V$.

\subsubsection{ Proof of Theorem 3, (i), (ii), (iii)}

We start with the relations $u\left( a\right) ^{3}=u\left( b\right)
^{3}=u\left( ab\right) ^{3}=u(a^{-1}b)^{3}=u(a^{2}b)^{3}=u(ab^{2})^{3}=0$
which translate to $U\ ^{3}=V^{3}=0,(UV+U+V)^{3}=0$. The interchange $%
a\leftrightarrow b$ translates to $U\leftrightarrow V$ , the substitution $%
ab\rightarrow a^{-1}b$ translates to $U$ $\rightarrow U^{2}-U,V\rightarrow V$
and similarly, $ab\rightarrow a^{2}b$ translates to $U$ $\rightarrow
U^{2}+2U,$ $V\rightarrow V$.

We also order $U<V<U^{2}<V^{2}$. Monomials are ordered decreasingly as
follows: by syllable length, then by total length, then by total $V$ length
and finally, by reverse lexicographical ordering.

In the sequence of manipulations below we use the following notation: given $%
m,p\in \mathbf{B}$ where $m$ is a monomial and where the monomial summands
of $p$ are of lesser order than $p$, a relation (or, rule of substituting $m$
by $p$) is of the form $m=p$. Given a monomial $q\in \mathbf{B}$ written as
a sum of monomials $q=\cdots +c\ast m+\cdots $ $\left( c\in \mathbb{Z}%
\right) $ then $q\leftarrow m=p$ (abbreviated as $q\leftarrow m)$ indicates
the substitution of $m$ by $p$ in the expression for $q$. If $q=\cdots
+c\ast m^{\prime }\ast m\ast m^{\prime \prime }+\cdots $ where one of the
monomials $m^{\prime },m^{\prime \prime }$ is non-empty, then $%
q\leftleftarrows m=p$ (abbreviated as $q\leftleftarrows m)$ indicates the
substitution of $m$ by $p$ in $q$. Composing substitutions is done from left
to right and we write $q\leftarrow m_{1},m_{2}$ for $\left( q\leftarrow
m_{1}\right) \leftarrow m_{2}$.

\begin{enumerate}
\item := $(UV + U + V)^3=0$ \label{eq:UVUVUV}\newline
$UVUVUV \;=\; - VUVUV - UVUVU - UVUV^2 - UV^2UV - UVU^2V - U^2VUV - 2UVUV -
VUVU - VUV^2 - V^2UV - VU^2V - UV^2U - UVU^2 - U^2VU - VUV - UVU - U^2V^2 -
UV^2 - V^2U - U^2V - VU^2 $;\newline
\newline

\item := 
\textcolor{red}{$ (1) \times V^2 \leftleftarrows (1) \leftarrow (1)
$} \label{eq:UVUVU}\newline
$UVUVU \;=\; - UVUV - VUVU - UV^2U - UVU^2 - U^2VU - VUV - UVU - UV^2 - V^2U
- U^2V - VU^2 $;\newline
\newline

\item := 
\textcolor{red}{$ (2) \times U^2  \leftleftarrows (2) \leftarrow
(2)$} \label{eq:UVUV}\newline
$UVUV \;=\; - VUV - UVU + V^2U^2 - UV^2 - V^2U - U^2V - VU^2 $;\newline
\newline

\item := 
\textcolor{red}{$ U^2 \times (3)  \leftleftarrows (3) \leftarrow
(3)$} \label{eq:U2V2U2}\newline
$U^2V^2U^2 \;=\; UV^2U^2 + U^2V^2U + U^2VU^2 - UV^2U - UVU^2 - U^2VU + VUV +
UVU - U^2V^2 - V^2U^2 + UV^2 + V^2U + U^2V + VU^2 $;\newline
\newline

\item := \textcolor{red}{$ (4) \leftarrow  (U \mapsto U^2 -U)  $} \label%
{eq:VU2V}\newline
$VU^2V \;=\; UV^2U + VUV - UVU + UV^2 + V^2U - U^2V - VU^2 $;\newline
\newline

\item := \textcolor{red}{$ (4) \leftarrow  (V \mapsto V^2 -V)  $} \label%
{eq:V2UV2}\newline
$V^2UV^2 \;=\; U^2VU^2 + VUV^2 + V^2UV - UVU^2 - U^2VU - VUV + UVU - UV^2 -
V^2U + U^2V + VU^2 $;\newline
\newline

\item := 
\textcolor{red}{$ (4) \leftarrow  (V \mapsto V^2 +2V) \leftarrow
(6) \leftarrow (4) \leftarrow * \frac{1}{3}$} \label{eq:VUV2}\newline
$VUV^{2}\;=\;-V^{2}UV$;\newline
\newline

\item := \textcolor{red}{$ (7) \times V $ } \label{eq:V2UV2a}\newline
$V^{2}UV^{2}\;=\;0$;\newline
\newline

\item := \textcolor{red}{$ (7) \leftarrow (U \leftrightarrow V)  $} \label%
{eq:UVU2}\newline
$UVU^2 \;=\; - U^2VU $;\newline
\newline

\item := \textcolor{red}{$ (8) \leftarrow (U \leftrightarrow V)  $} \label%
{eq:U2VU2}\newline
$U^{2}VU^{2}\;=\;0$;\newline
\newline

\item := 
\textcolor{red}{$ (6)   \leftarrow  (10) \leftarrow (9) \leftarrow 
\leftarrow (8) \leftarrow (7) $} \label{eq:VUV}\newline
$VUV \;=\; UVU - UV^2 - V^2U + U^2V + VU^2 $;\newline
\newline

\item := \textcolor{red}{$ (5)  \leftarrow (11)$} \label{eq:VU2Va}\newline
$VU^{2}V\;=\;UV^{2}U$;\newline
\newline

\item := 
\textcolor{red}{$ (11) \times V \leftarrow (12) \leftarrow (7)
\leftarrow (3) \leftarrow (11) $} \label{eq:UV2U}\newline
$UV^{2}U\;=\;2UVU-U^{2}V^{2}-V^{2}U^{2}+2U^{2}V+2VU^{2}$.\newline
\newline

\item := 
\textcolor{red}{$ (13) \times U \leftleftarrows (13) \leftarrow (9)
$} \label{eq:UV2U2}\newline
$UV^{2}U^{2}=-U^{2}V^{2}U$ \newline

\item := \textcolor{red}{$ U \times (14)  $}\newline
$U^{2}V^{2}U^{2}=0$.
\end{enumerate}

Further similar manipulations (see \cite{GOS}) produce%
\begin{equation*}
\left( UV\right) ^{3}=\left( UV\right) ^{2}U^{2}=U^{2}\left( VU\right)
^{2}=\left( UV^{2}\right) ^{2}=\left( U^{2}V\right) ^{2}=\left( UVU\right)
^{2}=0\text{;}
\end{equation*}%
\begin{eqnarray*}
VUVU^{2}V &=&UV^{2}UVU=UVU^{2}V^{2}=U^{2}V^{2}UV=V^{2}UVU^{2}=VU^{2}V^{2}U%
\text{;} \\
V^{2}U^{2}VU
&=&V^{2}UVU^{2}=VUV^{2}U^{2}=UV^{2}U^{2}V=U^{2}VUV^{2}=UVUV^{2}U=VU^{2}VUV=-VU^{2}V^{2}U%
\text{.}
\end{eqnarray*}%
With these relations, it is directly verifiable that monomials of length $7$
are null.

\begin{lemma}
Let $c_{2}=\left[ a,b\right] $ and $C_{2}=u\left( c_{2}\right) =c_{2}-1$.
Then $C_{2}^{3}=6UV^{2}U^{2}V=0$.
\end{lemma}

\begin{proof}
Since for a general element $g$ of the group, $g^{-1}=u\left( g\right)
^{2}-u\left( g\right) +1$, 
\begin{equation*}
C_{2}=(U^{2}-U+1)(V^{2}-V+1)(U+1)(V+1)-1\text{.}
\end{equation*}%
On expanding the above expression and eliminating monomials of length $7$ we
get 
\begin{eqnarray*}
C_{2}^{3} &=&\left( UV\right) ^{3}-\left( VU\right) ^{3} \\
&&+UV^{2}UVU-UVUV^{2}U \\
&&+VUVU^{2}V-VU^{2}VUV \\
&&+VU^{2}V^{2}U-UV^{2}U^{2}V\text{.}
\end{eqnarray*}%
Then, with further easy substitutions from the relations listed above, we
obtain 
\begin{equation*}
C_{2}^{3}=6U^{2}V^{2}UV\text{.}
\end{equation*}
\end{proof}

\subsubsection{Table of relations for the augmentation ideal $\mathbf{B}$}

\begin{eqnarray}
VUV &=&UVU-UV^{2}+U^{2}V-V^{2}U+VU^{2}, \\
\left( UV\right) ^{2} &=&-2UVU-2U^{2}V+V^{2}U^{2}-2VU^{2}, \\
\left( VU\right) ^{2} &=&-2UVU+U^{2}V^{2}-2U^{2}V-2VU^{2}, \\
\left( VU\right) ^{2}-\left( UV\right) ^{2} &=&U^{2}V^{2}-V^{2}U^{2}, \\
UV^{2}U &=&2\left( UVU+U^{2}V+VU^{2}\right) -U^{2}V^{2}-V^{2}U^{2}, \\
\left( UV\right) ^{2}+UV^{2}U &=&-U^{2}V^{2}, \\
VU^{2}V &=&UV^{2}U, \\
UVU^{2} &=&-U^{2}VU, \\
VUV^{2} &=&-V^{2}UV, \\
U^{2}VU^{2} &=&0,V^{2}UV^{2}=0, \\
UVUVU &=&0, \\
VU^{2}V^{2}U &=&U^{2}VU^{2}V=0\text{.} \\
V^{2}UVU &=&UV^{2}UV=-V^{2}U^{2}V, \\
VUV^{2}U &=&UVUV^{2}=V^{2}U^{2}V, \\
VU^{2}VU &=&U^{2}VUV=-U^{2}V^{2}U, \\
VUVU^{2} &=&UVU^{2}V=U^{2}V^{2}U\text{.}
\end{eqnarray}

Analysis of the interdependence of the relations in $\mathbf{B}$ produces
the following (see, \cite{GOS}),

\begin{lemma}
\bigskip The ring $\mathbf{B}$ is defined over $\mathbb{D}$ by its
generators $U,V$ with relations%
\begin{eqnarray*}
VUV &=&UVU-UV^{2}+U^{2}V-V^{2}U+VU^{2}, \\
UV^{2}U &=&2\left( UVU+U^{2}V+VU^{2}\right) -U^{2}V^{2}-V^{2}U^{2}, \\
VU^{2}V &=&UV^{2}U, \\
UVU^{2} &=&-U^{2}VU, \\
VU^{2}V^{2}U &=&U^{2}VU^{2}V=0\text{.}
\end{eqnarray*}
\end{lemma}

The rank of the ring $\mathbf{A}$ is now at most $18$, and is $\mathbb{D}$%
-generated by 
\begin{eqnarray*}
&&\{1,U,V,U^{2},V^{2},UV,VU, \\
&&U^{2}V,VU^{2},UV^{2},V^{2}U,V^{2}U^{2},U^{2}V^{2}, \\
&&VUV,U^{2}VU,V^{2}UV, \\
&&U^{2}V^{2}U,V^{2}U^{2}V\}\text{.}
\end{eqnarray*}

On rewriting the generating set as $\ $%
\begin{eqnarray*}
&&%
\{1,U,V,U^{2},UV,VU,V^{2},U^{2}V,UV^{2},VU^{2},V^{2}U,U^{2}VU,U^{2}V^{2},VU^{2}V,V^{2}U^{2},V^{2}UV,
\\
&&U^{2}V^{2}U,V^{2}U^{2}V\}
\end{eqnarray*}%
we find that $a,b$ are represented by the following upper triangular matrices

\begin{center}
$a\mapsto \left( 
\begin{array}{cccccccccccccccccc}
1 & 1 & 0 & 0 & 0 & 0 & 0 & 0 & 0 & 0 & 0 & 0 & 0 & 0 & 0 & 0 & 0 & 0 \\ 
0 & 1 & 0 & 1 & 0 & 0 & 0 & 0 & 0 & 0 & 0 & 0 & 0 & 0 & 0 & 0 & 0 & 0 \\ 
0 & 0 & 1 & 0 & 0 & 1 & 0 & 0 & 0 & 0 & 0 & 0 & 0 & 0 & 0 & 0 & 0 & 0 \\ 
0 & 0 & 0 & 1 & 0 & 0 & 0 & 0 & 0 & 0 & 0 & 0 & 0 & 0 & 0 & 0 & 0 & 0 \\ 
0 & 0 & 0 & 0 & 1 & 0 & 0 & -1 & 0 & -1 & 0 & 0 & \frac{1}{2} & \frac{1}{2}
& \frac{1}{2} & 0 & 0 & 0 \\ 
0 & 0 & 0 & 0 & 0 & 1 & 0 & 0 & 0 & 1 & 0 & 0 & 0 & 0 & 0 & 0 & 0 & 0 \\ 
0 & 0 & 0 & 0 & 0 & 0 & 1 & 0 & 0 & 0 & 1 & 0 & 0 & 0 & 0 & 0 & 0 & 0 \\ 
0 & 0 & 0 & 0 & 0 & 0 & 0 & 1 & 0 & 0 & 0 & 1 & 0 & 0 & 0 & 0 & 0 & 0 \\ 
0 & 0 & 0 & 0 & 0 & 0 & 0 & 0 & 1 & 0 & 0 & 0 & 0 & 1 & 0 & 0 & 0 & 0 \\ 
0 & 0 & 0 & 0 & 0 & 0 & 0 & 0 & 0 & 1 & 0 & 0 & 0 & 0 & 0 & 0 & 0 & 0 \\ 
0 & 0 & 0 & 0 & 0 & 0 & 0 & 0 & 0 & 0 & 1 & 0 & 0 & 0 & 1 & 0 & 0 & 0 \\ 
0 & 0 & 0 & 0 & 0 & 0 & 0 & 0 & 0 & 0 & 0 & 1 & 0 & 0 & 0 & 0 & 0 & 0 \\ 
0 & 0 & 0 & 0 & 0 & 0 & 0 & 0 & 0 & 0 & 0 & 0 & 1 & 0 & 0 & 0 & 1 & 0 \\ 
0 & 0 & 0 & 0 & 0 & 0 & 0 & 0 & 0 & 0 & 0 & 0 & 0 & 1 & 0 & 0 & -1 & 0 \\ 
0 & 0 & 0 & 0 & 0 & 0 & 0 & 0 & 0 & 0 & 0 & 0 & 0 & 0 & 1 & 0 & 0 & 0 \\ 
0 & 0 & 0 & 0 & 0 & 0 & 0 & 0 & 0 & 0 & 0 & 0 & 0 & 0 & 0 & 1 & 0 & -1 \\ 
0 & 0 & 0 & 0 & 0 & 0 & 0 & 0 & 0 & 0 & 0 & 0 & 0 & 0 & 0 & 0 & 1 & 0 \\ 
0 & 0 & 0 & 0 & 0 & 0 & 0 & 0 & 0 & 0 & 0 & 0 & 0 & 0 & 0 & 0 & 0 & 1%
\end{array}%
\right) $\bigskip \linebreak $b\mapsto \left( 
\begin{array}{cccccccccccccccccc}
1 & 0 & 1 & 0 & 0 & 0 & 0 & 0 & 0 & 0 & 0 & 0 & 0 & 0 & 0 & 0 & 0 & 0 \\ 
0 & 1 & 0 & 0 & 1 & 0 & 0 & 0 & 0 & 0 & 0 & 0 & 0 & 0 & 0 & 0 & 0 & 0 \\ 
0 & 0 & 1 & 0 & 0 & 0 & 1 & 0 & 0 & 0 & 0 & 0 & 0 & 0 & 0 & 0 & 0 & 0 \\ 
0 & 0 & 0 & 1 & 0 & 0 & 0 & 1 & 0 & 0 & 0 & 0 & 0 & 0 & 0 & 0 & 0 & 0 \\ 
0 & 0 & 0 & 0 & 1 & 0 & 0 & 0 & 1 & 0 & 0 & 0 & 0 & 0 & 0 & 0 & 0 & 0 \\ 
0 & 0 & 0 & 0 & 0 & 1 & 0 & 0 & -1 & 0 & -1 & 0 & \frac{1}{2} & \frac{1}{2}
& \frac{1}{2} & 0 & 0 & 0 \\ 
0 & 0 & 0 & 0 & 0 & 0 & 1 & 0 & 0 & 0 & 0 & 0 & 0 & 0 & 0 & 0 & 0 & 0 \\ 
0 & 0 & 0 & 0 & 0 & 0 & 0 & 1 & 0 & 0 & 0 & 0 & 1 & 0 & 0 & 0 & 0 & 0 \\ 
0 & 0 & 0 & 0 & 0 & 0 & 0 & 0 & 1 & 0 & 0 & 0 & 0 & 0 & 0 & 0 & 0 & 0 \\ 
0 & 0 & 0 & 0 & 0 & 0 & 0 & 0 & 0 & 1 & 0 & 0 & 0 & 1 & 0 & 0 & 0 & 0 \\ 
0 & 0 & 0 & 0 & 0 & 0 & 0 & 0 & 0 & 0 & 1 & 0 & 0 & 0 & 0 & 1 & 0 & 0 \\ 
0 & 0 & 0 & 0 & 0 & 0 & 0 & 0 & 0 & 0 & 0 & 1 & 0 & 0 & 0 & 0 & -1 & 0 \\ 
0 & 0 & 0 & 0 & 0 & 0 & 0 & 0 & 0 & 0 & 0 & 0 & 1 & 0 & 0 & 0 & 0 & 0 \\ 
0 & 0 & 0 & 0 & 0 & 0 & 0 & 0 & 0 & 0 & 0 & 0 & 0 & 1 & 0 & 0 & 0 & -1 \\ 
0 & 0 & 0 & 0 & 0 & 0 & 0 & 0 & 0 & 0 & 0 & 0 & 0 & 0 & 1 & 0 & 0 & 1 \\ 
0 & 0 & 0 & 0 & 0 & 0 & 0 & 0 & 0 & 0 & 0 & 0 & 0 & 0 & 0 & 1 & 0 & 0 \\ 
0 & 0 & 0 & 0 & 0 & 0 & 0 & 0 & 0 & 0 & 0 & 0 & 0 & 0 & 0 & 0 & 1 & 0 \\ 
0 & 0 & 0 & 0 & 0 & 0 & 0 & 0 & 0 & 0 & 0 & 0 & 0 & 0 & 0 & 0 & 0 & 1%
\end{array}%
\text{.}\right) $\bigskip \linebreak
\end{center}

With the use of GAP \cite{gap}, it is shown that the matrices satisfy the
relations and thus $rank_{\mathbb{D}}\left( \mathbf{A}\right) =18$, $\mathbf{%
B}$ is nilpotent of degree $6$ and the group $G$ is nilpotent of degree $4$.

To prove that $G$ is free $2$-generated of nilpotency class $4$, we find \
the expressions for the basic commutators in terms of the basis elements of $%
\mathbf{B}$ \cite{GOS}:

\begin{enumerate}
\item $[a,b]-1=UV-VU-6U^{2}V+UV^{2}-5VU^{2}+2V^{2}U-4UVU$\newline
$+U^{2}V^{2}+2V^{2}U^{2}-U^{2}VU-VUV^{2}+2U^{2}V^{2}U-VU^{2}V^{2},$

\item $[a,b,a]-1=-3U^{2}V-3VU^{2}+3V^{2}U^{2}-6U^{2}VU+3U^{2}V^{2}U,$

\item $[a,b,b]-1=3UV^{2}+3V^{2}U-3U^{2}V^{2}+6V^{2}UV-3V^{2}U^{2}V,$

\item $[a,b,a,a]-1=-6U^{2}VU+3U^{2}V^{2}U,$

\item $[a,b,b,a]-1=[a,b,a,b]-1=-3U^{2}V^{2}+3V^{2}U^{2},$

\item $[a,b,b,b]-1=6V^{2}UV-3V^{2}U^{2}V.$

\item $[a,b,\ast ,\ast ,\ast ]=1$.
\end{enumerate}

Since the set of monomials of minimal weight in the above commutators form
an independent set, it follows that $G$ is free nilpotent of class $4$.

\subsection{Proof of Theorem 3 (iv)}

\begin{lemma}
The set of solutions of $X^{3}=0$ in $\mathbf{B}$ are described by the
algebraic affine variety $\mathcal{V}$ over $\mathbb{D}$, given by equation%
\begin{equation*}
\frac{x_{1}x_{2}(x_{1}+x_{2})}{2}%
-x_{1}x_{2}x_{3}-x_{1}x_{2}x_{4}+x_{2}^{2}x_{5}+x_{1}^{2}x_{6}=0.\,\,(\ast )%
\text{.}
\end{equation*}
\end{lemma}

\begin{proof}
Let%
\begin{equation*}
X=x_{1}V+x_{2}U+x_{3}VU+x_{4}UV+x_{5}V^{2}+x_{6}U^{2}+x_{7}V^{2}U+x_{8}VU^{2}+x_{9}UV^{2}+x_{10}U^{2}V+W_{0},
\end{equation*}%
where $W_{0}\in \mathbf{B}^{4}$. Since $\mathbf{B}^{6}=0$, we deduce that $%
X^{3}=\sum x_{i}x_{j}x_{k}f_{ijk}$ where $1\leq i\leq j\leq k\leq 10$ and no
two indices belong to $\left\{ 7,8,9,10\right\} $. Furthermore, $f_{ijk}=0$
for $i=1,2;j,k=3,...,6$; for example, 
\begin{eqnarray*}
f_{134}
&=&V^{2}U^{2}V+VU^{2}V^{2}+VUVUV+UV^{2}UV+VUV^{2}U=V^{2}U^{2}V+VU^{2}V^{2}+UV^{2}UV+VUV^{2}U%
\text{ (by(4.11)) } \\
&=&V^{2}U^{2}V+2UV^{2}UV+VUV^{2}U\text{ (by(4.7))}=2V^{2}U^{2}V+2UV^{2}UV%
\text{ (by(4.14))}=0\text{ (by(4.15)).}
\end{eqnarray*}

The expression for $X^{3}$ becomes

$%
X^{3}=x_{1}^{2}x_{2}(V^{2}U+VUV+UV^{2})+x_{1}x_{2}^{2}(VU^{2}+UVU+U^{2}V)+x_{1}^{2}x_{3}(V^{2}UV+VUV^{2})+x_{1}x_{2}x_{3}(V^{2}U^{2}+UVUV+2VU^{2}V+2VUVU)+ 
$

$x_{2}^{2}x_{3}(U^{2}VU+UVU^{2})+$

$%
x_{1}^{2}x_{4}(V^{2}UV+VUV^{2})+x_{1}x_{2}x_{4}(U^{2}V^{2}+VUVU+2UV^{2}U+2UVUV)+ 
$

$x_{2}^{2}x_{4}(U^{2}VU+UVU^{2})+$

$%
x_{1}^{2}x_{6}(V^{2}U^{2}+VU^{2}V+U^{2}V^{2})+x_{1}x_{2}x_{5}(V^{2}UV+VUV^{2})+ 
$

$%
x_{2}^{2}x_{5}(U^{2}V^{2}+V^{2}U^{2}+UV^{2}U)+x_{1}x_{2}x_{6}(U^{2}VU+UVU^{2})+ 
$

$%
x_{1}^{2}x_{7}(V^{2}UV^{2})+x_{1}x_{2}x_{7}(UV^{2}UV+V^{2}U^{2}V+V^{2}UVU+VUV^{2}U)+ 
$

$x_{2}^{2}x_{7}(U^{2}V^{2}U+UV^{2}U^{2})+$

$%
x_{1}^{2}x_{8}(V^{2}U^{2}V+VU^{2}V^{2})+x_{1}x_{2}x_{8}(VU^{2}VU+UV^{2}U^{2}+VUVU^{2}+UVU^{2}V)+ 
$

$x_{2}^{2}x_{8}(U^{2}VU^{2})+$

$%
x_{1}^{2}x_{9}(U^{2}V^{2}U+UV^{2}U^{2})+x_{1}x_{2}x_{9}(UV^{2}UV+VU^{2}V^{2}+UVUV^{2}+VUV^{2}U)+ 
$

$x_{2}^{2}x_{9}(V^{2}UV^{2})+$

$%
x_{1}^{2}x_{10}(U^{2}VU^{2})+x_{1}x_{2}x_{10}(VU^{2}VU+U^{2}V^{2}U+U^{2}VUV+UVU^{2}V)+ 
$

$x_{2}^{2}x_{10}(V^{2}U^{2}V+VU^{2}V^{2})+$

$%
x_{1}x_{3}^{2}(V^{2}UVU+VUVUV+VUV^{2}U)+x_{2}x_{3}^{2}(UVUVU+VU^{2}VU+VUVU^{2})+ 
$

$%
x_{1}x_{3}x_{4}(V^{2}U^{2}V+VU^{2}V^{2}+VUVUV+UV^{2}UV+VUV^{2}U)+x_{2}x_{5}^{2}V^{2}UV^{2}+ 
$

$%
x_{1}x_{4}^{2}(UV^{2}UV+VUVUV+UVUV^{2})+x_{2}x_{4}^{2}(U^{2}VUV+UVUVU+UVU^{2}V)+ 
$

$%
x_{1}x_{3}x_{5}V^{2}UV^{2}+x_{2}x_{3}x_{5}(UVUV^{2}+VU^{2}V^{2}+V^{2}UVU+VUV^{2}U)+ 
$

$%
x_{1}x_{4}x_{5}V^{2}UV^{2}+x_{2}x_{4}x_{5}(UV^{2}UV+UVUV^{2}+V^{2}U^{2}V+V^{2}UVU)+ 
$

$x_{1}x_{5}x_{6}(V^{2}U^{2}V+VU^{2}V^{2})$.

On using our table of relations, this expression reduces to%
\begin{equation*}
X^{3}=\left(
x_{1}^{2}x_{2}+x_{1}x_{2}^{2}-2x_{1}x_{2}x_{4}-2x_{1}x_{2}x_{3}+2x_{1}^{2}x_{6}+2x_{2}^{2}x_{5}\right) (UVU+U^{2}V+VU^{2})%
\text{.}
\end{equation*}
\end{proof}

\begin{lemma}
In the ring $\mathbf{A}$, we have $X^{3}=0$ for every $1+X\in G=<1+U,1+V>$.
\end{lemma}

\begin{proof}
It is easy to see that $I=\mathbf{B}$ is a group with respect to the
operation $a\circ b=a+b+ab$. Unfortunately, $\mathcal{V}$ is not a subgroup
of $I$. The following steps lead to $G\subset \mathcal{V}$.

\textbf{Step 1.} For any $n,m\in \mathbf{N}$ 
\begin{equation*}
(1+U)^{n}(1+V)^{m}=1+S,S\in \mathcal{V}.
\end{equation*}

\textbf{Step 2.} If $1+C\in \lbrack G,G],$ then 
\begin{equation*}
C=y(VU-UV)+C_{3},C_{3}\in I^{3}.
\end{equation*}%
\textbf{Step 3.} If $X\in \mathcal{V},$ $C\in \lbrack G,G],$ then $X\circ
C\in \mathcal{V}$. Indeed, let $%
X=x_{1}V+x_{2}U+x_{3}VU+x_{4}UV+x_{5}V^{2}+x_{6}U^{2}+W,$ $W\in \mathbf{B}%
^{3}$. Then%
\begin{eqnarray*}
X\circ C &=&X+C+XC \\
&=&x_{1}V+x_{2}U+(x_{3}+y)VU+(x_{4}-y)UV+x_{5}V^{2}+x_{6}U^{2}+W_{1}, \\
W_{1} &\in &\mathbf{B}^{3}
\end{eqnarray*}%
and the coefficients of this element satisfy the equation (*).

Finally, since $\mathbf{B}$ is nilpotent, every element of $G$ has the form $%
(1+U)^{n}(1+V)^{m}(1+C),1+C\in \lbrack G,G]$ and the proof follows.
\end{proof}

\subsection{The $d$-generated case}

Recall $F_{d}$ is the free group of rank $d$ generated by $\left\{
y_{1},...,y_{d}\right\} $. Let $x_{i}\in \left\{ y_{1}^{\pm
1},...,y_{d}^{\pm 1}\right\} $ $\left( 1\leq i\leq m\right) $ and $%
r=x_{1}x_{2}...x_{m}$ be a reduced word. Then $\left\{
x_{1},x_{1}x_{2},...,\left( x_{1}x_{2}...x_{i}\right) ,..,r\right\} $ is the
set of initial segments of $r$.

The following formula holds in the augmentation ideal of $\mathbb{Z}F_{d}$, 
\begin{equation*}
u\left( r\right) =u\left( x_{1}...x_{m}\right) =\sum_{1\leq
i_{1}<...<i_{t}<m}u\left( x_{i_{1}}\right) ...u\left( x_{i_{t}}\right)
+u\left( x_{1}\right) ...u\left( x_{m}\right) \text{.}
\end{equation*}%
The monomial $\nu =u\left( x_{1}\right) ...u\left( x_{m}\right) $ in the
above expression is the leading term and it is of highest weight. On the
other hand, given a monomial $\nu $, there exists $r\in F_{d}$, as above,
for which $\nu $ is a leading monomial of $u\left( r\right) $; write $%
r=r_{\nu }\left( y_{i}\right) $.

As we assume that $u\left( a_{i}\right) ^{3}=0$ for all $i$, with the
notation $u\left( a_{i}\right) =U_{i}$, our monomials in the quotient ring
are of the form $\nu =U_{i_{1}}^{\varepsilon _{1}}U_{i_{2}}^{\varepsilon
_{2}}...U_{i_{s}}^{\varepsilon _{s}}$ where $U_{i_{t}}\not=U_{i_{t+1}}$ and $%
\varepsilon _{i}=1,2$. Furthermore, since $u\left( a_{i}^{-1}\right)
=-U_{i}+U_{i}^{2}$, we can choose $r_{\nu }=r_{\nu }\left( a_{i}\right) $ to
be $a_{i_{1}}^{\delta _{1}}a_{i_{2}}^{\delta _{2}}...a_{i_{s}}^{\delta _{s}}$
where $a_{i_{t}}\not=a_{i_{t+1}}$ , $\delta _{i}=1$ for $\varepsilon _{i}=1$
and $\delta _{i}=-1$ for $\varepsilon _{i}=2$.

We denote the augmentation $\omega \left( \mathbf{A}_{d}\right) $ by $\omega
_{d}$.

\subsubsection{Proof of Theorem $4$}

For $d=1$, we take $\mathbf{A}_{1}=\frac{\mathbb{D}F_{1}}{\left\langle
\left( x-1\right) ^{3}\mid x\in P_{1}\right\rangle }$ where $P_{1}=\left\{
y_{1}\right\} $; the augmentation ideal $\omega _{1}$ is the $\mathbb{D}$%
-span of $\mathbf{W}_{\mathbf{1}}$ $\mathbf{=}$ $\left\{
U_{1},U_{1}^{2}\right\} $. For $d=2$, we take $\mathbf{A}_{2}=\mathbf{A}=%
\frac{\mathbb{D}F_{2}}{\left\langle \left( x-1\right) ^{3}\mid x\in
P\right\rangle }$ where $P=\left\{
y_{1},y_{2},y_{1}y_{2},y_{1}^{-1}y_{2},y_{1}^{2}y_{2},y_{1}y_{2}^{2}\right\} 
$. We recall that the following relations which hold in $\omega _{2}=\omega
\left( \mathbf{A}_{2}\right) $, 
\begin{eqnarray*}
U_{2}U_{1}U_{2}^{2} &=&-U_{2}^{2}U_{1}U_{2} \\
U_{2}U_{1}U_{2}
&=&U_{1}U_{2}U_{1}-U_{1}U_{2}^{2}+U_{1}^{2}U_{2}-U_{2}^{2}U_{1}+U_{2}U_{1}^{2}%
\text{.}
\end{eqnarray*}%
We conclude 
\begin{equation*}
U_{2}\mu U_{2}^{2}=-U_{2}^{2}\mu U_{2},U_{2}^{2}\mu U_{2}^{2}=0
\end{equation*}%
and $U_{2}\mu U_{2}$ is a sum of monomials $\mu _{1}U_{2}\mu _{2}$, $\mu
_{1}^{\prime }U_{2}^{2},U_{2}^{2}\mu _{2}^{\prime }$ where $\mu _{1},\mu
_{2},\mu _{1}^{\prime },\mu _{2}^{\prime }$ are monomials in $\omega _{1}$.
Therefore $\omega _{2}$ is $\mathbb{D}$-generated by the monomials 
\begin{eqnarray*}
\mathbf{W}_{2} &\mathbf{=}&\mathbf{W}_{\mathbf{1}}\cup \left\{
U_{2}^{\varepsilon }\right\} \cup \mathbf{W}_{\mathbf{1}}U_{2}^{\varepsilon
}\cup U_{2}^{\varepsilon }\mathbf{W}_{\mathbf{1}} \\
&&\cup U_{2}^{2}\mathbf{W}_{\mathbf{1}}\text{ }U_{2}\text{ }\cup U_{2}^{2}%
\mathbf{W}_{\mathbf{1}}U_{2}\mathbf{W}_{\mathbf{1}}\text{ } \\
&&\cup \text{ }\mathbf{W}_{\mathbf{1}}U_{2}^{2}\mathbf{W}_{\mathbf{1}}U_{2}%
\text{ } \\
&&\cup \text{ }\mathbf{W}_{\mathbf{1}}U_{2}^{2}\mathbf{W}_{\mathbf{1}}U_{2}%
\mathbf{W}_{\mathbf{1}}
\end{eqnarray*}
where $\varepsilon =1,2$.

We note that $\mathbf{W}_{2}$ is somewhat larger than we had produced
earlier. However, $\mathbf{W}_{2}$ lends itself to easier generalization.

Suppose $\mathbf{A}_{d}$ is defined such that $\omega _{d}$ is generated by
a finite set of non-zero monomials $\mathbf{W}_{\mathbf{d}}$ $\mathbf{=}$ $%
\left\{ \nu _{1}\left( a_{i}\right) ,...,\nu _{t}\left( a_{i}\right)
\right\} $. Define $\widetilde{P_{d}}$ to be the closure of%
\begin{equation*}
P_{d}\cup \left\{ r_{\nu _{1}}\left( y_{i}\right) ,...,r_{\nu _{t}}\left(
y_{i}\right) \right\}
\end{equation*}%
under taking initial segments in $F_{d}$. Then 
\begin{eqnarray*}
P_{d+1} &=&\widetilde{P_{d}} \\
&&\cup \left\{ e,h^{\pm 1}\mid h\in \widetilde{P_{d}}\right\} y_{d+1} \\
&&\cup \widetilde{P_{d}}\text{ }y_{d+1}^{2} \\
&&\cup \left\{ h^{2}\mid h\in \widetilde{P_{d}}\right\} y_{d+1}
\end{eqnarray*}%
and we define 
\begin{equation*}
\mathbf{A}_{d+1}=\frac{\mathbb{D}F_{d+1}}{\left\langle \left( x-1\right)
^{3}\mid x\in P_{d+1}\right\rangle }\text{.}
\end{equation*}

We will write $b=a_{d+1}$ in the ring $\mathbf{A}_{d+1}=\frac{\mathbb{D}F_{d}%
}{\left\langle \left( x-1\right) ^{3}\mid x\in P_{d+1}\right\rangle }$ and
write $V=U_{d+1}$. Thus, for every $q\in \widetilde{P_{d}}$ the subring $%
\left\langle q,b\right\rangle $ of $\mathbf{A}_{d+1}$ is a homomorphic image
of $\mathbf{A}_{2}$ via $a\rightarrow q,b\rightarrow b$. Using the
epimorphism produces 
\begin{equation*}
Vu\left( q\right) V^{2}=-V^{2}u\left( q\right) V
\end{equation*}%
which generalizes to%
\begin{equation*}
V\mu V^{2}=-V^{2}\mu V,V^{2}\mu V^{2}=0
\end{equation*}%
for all $\mu $ monomial in the subring $\left\langle
U_{1},...,U_{d}\right\rangle $. Furthermore, 
\begin{equation*}
Vu\left( q\right) V=u\left( q\right) Vu\left( q\right) -u\left( q\right)
V^{2}+u\left( q\right) ^{2}V-V^{2}u\left( q\right) +Vu\left( q\right) ^{2}%
\text{.}
\end{equation*}%
Let $q=q_{1}x$ where $q_{1}$ is an initial segment of $q$. Then $u\left(
q\right) =u\left( q_{1}\right) +u\left( x\right) +u\left( q_{1}\right)
u\left( x\right) $ and, as the above formula holds for $u\left( q_{1}\right)
,u\left( x\right) $, we expand our expression to obtain a formula for $%
Vu\left( q_{1}\right) u\left( x_{1}\right) V$.

To illustrate, we apply the substitution 
\begin{equation*}
a_{1}\rightarrow a_{1}a_{2},\text{ }b\rightarrow b
\end{equation*}%
which translates to the substitution 
\begin{eqnarray*}
U_{1} &\rightarrow &U_{1}+U_{2}+U_{1}U_{2}, \\
U_{1}^{2} &\rightarrow
&U_{2}^{2}U_{1}^{2}-U_{2}^{2}U_{1}+U_{2}^{2}-U_{2}U_{1}^{2}+U_{1}U_{2}+U_{2}U_{1}+U_{1}^{2}
\\
V &\rightarrow &V\text{.}
\end{eqnarray*}%
in $\omega _{2}$.

The equation 
\begin{equation*}
VUV=UVU-UV^{2}+U^{2}V-V^{2}U+VU^{2}
\end{equation*}%
transforms to 
\begin{eqnarray*}
&&V\left( U_{1}+U_{2}+U_{1}U_{2}\right) V \\
&=&\left( U_{1}+U_{2}+U_{1}U_{2}\right) V\left(
U_{1}+U_{2}+U_{1}U_{2}\right) -\left( U_{1}+U_{2}+U_{1}U_{2}\right) V^{2} \\
&&+\left(
U_{2}^{2}U_{1}^{2}-U_{2}^{2}U_{1}+U_{2}^{2}-U_{2}U_{1}^{2}+U_{1}U_{2}+U_{2}U_{1}+U_{1}^{2}\right) V
\\
&&-V^{2}\left( UVU-UV^{2}+U^{2}V-V^{2}U+VU^{2}\right) \\
&&+V\left(
U_{2}^{2}U_{1}^{2}-U_{2}^{2}U_{1}+U_{2}^{2}-U_{2}U_{1}^{2}+U_{1}U_{2}+U_{2}U_{1}+U_{1}^{2}\right)
\end{eqnarray*}%
which, in expanded form, is%
\begin{eqnarray*}
&&VU_{1}V+VU_{2}V+VU_{1}U_{2}V \\
&=&U_{1}VU_{1}+U_{2}VU_{1}+U_{1}U_{2}VU_{1} \\
&&+U_{1}VU_{2}+U_{2}VU_{2}+U_{1}U_{2}VU_{2} \\
&&+U_{1}VU_{1}U_{2}+U_{2}VU_{1}U_{2}+U_{1}U_{2}VU_{1}U_{2} \\
&&-U_{1}V^{2}-U_{2}V^{2}-U_{1}U_{2}V^{2} \\
&&+U_{2}^{2}U_{1}^{2}V-U_{2}^{2}U_{1}V-U_{2}U_{1}^{2}V+U_{1}U_{2}V+U_{2}U_{1}V+U_{2}^{2}V+U_{1}^{2}V
\\
&&-V^{2}U_{1}-V^{2}U_{2}-V^{2}U_{1}U_{2} \\
&&+VU_{2}^{2}U_{1}^{2}-VU_{2}^{2}U_{1}-VU_{2}U_{1}^{2}+VU_{1}U_{2}+VU_{2}U_{1}+VU_{2}^{2}+VU_{1}^{2}%
\text{.}
\end{eqnarray*}%
We separate monomials involving just $U_{1}$ or $U_{2}$ in the above: 
\begin{eqnarray*}
VU_{1}U_{2}V &=&\left(
-VU_{1}V+U_{1}VU_{1}-U_{1}V^{2}-V^{2}U_{1}+VU_{1}^{2}+U_{1}^{2}V\right) \\
&&+\left(
-VU_{2}V+U_{2}VU_{2}-U_{2}V^{2}-V^{2}U_{2}+VU_{2}^{2}+U_{2}^{2}V\right) \\
&&+U_{2}VU_{1}+U_{1}U_{2}VU_{1} \\
&&+U_{1}VU_{2}+U_{1}U_{2}VU_{2} \\
&&+U_{1}VU_{1}U_{2}+U_{2}VU_{1}U_{2}+U_{1}U_{2}VU_{1}U_{2} \\
&&-U_{1}U_{2}V^{2} \\
&&+U_{2}^{2}U_{1}^{2}V-U_{2}^{2}U_{1}V-U_{2}U_{1}^{2}V+U_{1}U_{2}V+U_{2}U_{1}V
\\
&&-V^{2}U_{1}U_{2} \\
&&+VU_{2}^{2}U_{1}^{2}-VU_{2}^{2}U_{1}-VU_{2}U_{1}^{2}+VU_{1}U_{2}+VU_{2}U_{1}%
\text{.}
\end{eqnarray*}%
We separate monomials involving just $U_{1}$ or $U_{2}$ in the above: 
\begin{eqnarray*}
VU_{1}U_{2}V &=&\left(
-VU_{1}V+U_{1}VU_{1}-U_{1}V^{2}-V^{2}U_{1}+VU_{1}^{2}+U_{1}^{2}V\right) \\
&&+\left(
-VU_{2}V+U_{2}VU_{2}-U_{2}V^{2}-V^{2}U_{2}+VU_{2}^{2}+U_{2}^{2}V\right) \\
&&+U_{2}VU_{1}+U_{1}U_{2}VU_{1} \\
&&+U_{1}VU_{2}+U_{1}U_{2}VU_{2} \\
&&+U_{1}VU_{1}U_{2}+U_{2}VU_{1}U_{2}+U_{1}U_{2}VU_{1}U_{2} \\
&&-U_{1}U_{2}V^{2} \\
&&+U_{2}^{2}U_{1}^{2}V-U_{2}^{2}U_{1}V-U_{2}U_{1}^{2}V+U_{1}U_{2}V+U_{2}U_{1}V
\\
&&-V^{2}U_{1}U_{2} \\
&&+VU_{2}^{2}U_{1}^{2}-VU_{2}^{2}U_{1}-VU_{2}U_{1}^{2}+VU_{1}U_{2}+VU_{2}U_{1}%
\text{.}
\end{eqnarray*}%
Therefore, 
\begin{eqnarray*}
VU_{1}U_{2}V &=& \\
&&~U_{1}U_{2}VU_{1}U_{2} \\
&&+U_{1}U_{2}VU_{1}+U_{1}U_{2}VU_{2} \\
&&+U_{1}VU_{1}U_{2}+U_{2}VU_{1}U_{2} \\
&&+U_{1}U_{2}V+U_{1}VU_{2}+VU_{1}U_{2} \\
&&+U_{2}U_{1}V+U_{2}VU_{1}+VU_{2}U_{1} \\
&&\text{ }+U_{2}^{2}U_{1}^{2}V+VU_{2}^{2}U_{1}^{2}\text{ } \\
&&-U_{2}^{2}U_{1}V-VU_{2}^{2}U_{1}-U_{2}U_{1}^{2}V-VU_{2}U_{1}^{2} \\
&&-U_{1}U_{2}V^{2}-V^{2}U_{1}U_{2}\text{.}
\end{eqnarray*}%
The summands are of the form $\mu _{1}V\mu _{2}$, $\mu _{1}^{\prime
}V^{2},V^{2}\mu _{2}^{\prime }$ where $\mu _{1},\mu _{2},\mu _{1}^{\prime
},\mu _{2}^{\prime }$are monomials in the subring $\mathbf{C}%
_{2}=\left\langle U_{1},U_{2}\right\rangle $ and each of $l\left( \mu
_{1}\right) +l\left( \mu _{2}\right) ,l\left( \mu _{1}^{\prime }\right)
,l\left( \mu _{2}^{\prime }\right) $ is at least $l\left( U_{1}U_{2}\right)
=2$.

Thus, we derive the fact: for every $\mu $ monomial in the subring $\mathbf{C%
}_{d}=\left\langle U_{1},...,U_{d}\right\rangle $, of length $l\left( \mu
\right) =m$, the monomial $V\mu V$ decomposes as the sum of monomials of the
form $\mu _{1}V\mu _{2}$, $\mu _{1}^{\prime }V^{2},V^{2}\mu _{2}^{\prime }$
for some monomials $\mu _{1},\mu _{2},\mu _{1}^{\prime },\mu _{2}^{\prime }$
in the subring $\left\langle U_{1},...,U_{d}\right\rangle $ such that each
of $l\left( \mu _{1}\right) +l\left( \mu _{2}\right) ,l\left( \mu
_{1}^{\prime }\right) ,l\left( \mu _{2}^{\prime }\right) $ is at least $m$.

Given the above, we conclude that 
\begin{eqnarray*}
\mathbf{W}_{\mathbf{d+1}} &\mathbf{=}&\mathbf{W}_{\mathbf{d}}\cup \left\{
U_{d+1}^{\varepsilon }\right\} \cup \mathbf{W}_{\mathbf{d}%
}U_{d+1}^{\varepsilon }\cup U_{d+1}^{\varepsilon }\mathbf{W}_{\mathbf{d}} \\
&&\cup U_{d+1}^{2}\mathbf{W}_{\mathbf{d}}\text{ }U_{d+1}\text{ }\cup
U_{d+1}^{2}\mathbf{W}_{\mathbf{d}}U_{d+1}\mathbf{W}_{\mathbf{d}}\text{ } \\
&&\cup \text{ }\mathbf{W}_{\mathbf{d}}U_{d+1}^{2}\mathbf{W}_{\mathbf{d}%
}U_{d+1}\text{ } \\
&&\cup \text{ }\mathbf{W}_{\mathbf{d}}U_{d+1}^{2}\mathbf{W}_{\mathbf{d}%
}U_{d+1}\mathbf{W}_{\mathbf{d}}
\end{eqnarray*}%
where $\varepsilon =1,2$.

A monomial in $\omega _{d+1}$ has the form $W=\mu _{1}V^{i_{1}}\mu
_{2}V^{i_{2}}...\mu _{k}V^{i_{k}}$ where each $\mu _{i}\ $is a monomial in
the subring $\mathbf{C}_{d}$. Define $L\left( W\right) =\sum_{1\leq j\leq
k}l\left( \mu _{j}\right) $. Therefore $W$ is a sum of monomials of the form 
$\sigma V^{2s}\sigma ^{\prime }V^{t}\sigma ^{^{\prime \prime }}$ where $%
\sigma ,\sigma ^{\prime },\sigma ^{\prime \prime }$ are monomials in $%
\mathbf{C}_{d}$, with $s,t\in \left\{ 0,1\right\} $ and $l\left( \sigma
\right) +l\left( \sigma ^{\prime }\right) +l\left( \sigma ^{\prime \prime
}\right) \geq L\left( W\right) $.

Suppose the degree of nilpotency of $\omega _{d}$ is $\rho _{d}$ and $%
L\left( W\right) \geq 3\rho _{d}$. Then one of $l\left( \sigma \right)
,l\left( \sigma ^{\prime }\right) ,l\left( \sigma ^{\prime \prime }\right) $
is at least $\rho _{d}$ and therefore one of $\sigma ,\sigma ^{\prime
},\sigma ^{\prime \prime }$ is zero.

\end{document}